\documentclass[reqno,11pt]{article} %reqno is the option to put formula numbers on the right side
\usepackage{amsmath, amssymb, color}
\usepackage{amsthm} %theorem environment option
\usepackage{multicol}
\usepackage[pagewise]{lineno}%\linenumbers

\usepackage{mathtools}
\usepackage[normalem]{ulem}
\mathtoolsset{showonlyrefs}
\theoremstyle{plain} %text of this environment is typesetted in italics
\newtheorem{theorem}{\bf Theorem}[section]
\newtheorem{lemma}[theorem]{\bf Lemma}
\newtheorem{corollary}[theorem]{\bf Corollary}
\newtheorem{proposition}[theorem]{\bf Proposition}

\theoremstyle{definition} %text of this environment is typesetted in roman letters

\newtheorem{remark}[theorem]{\bf Remark}
\newtheorem{example}[theorem]{\bf Example}

\makeatletter
\def\address#1#2{\begingroup
	\noindent\parbox[t]{7.8cm}{%
		\small{\scshape\ignorespaces#1}\par\vskip1ex
		\noindent\small{\itshape E-mail address}%
		\/: #2\par\vskip4ex}\hfill%
	\endgroup}%
\makeatletter
\def\address#1#2{\begingroup
	\noindent\parbox[t]{7.8cm}{%
		\small{\scshape\ignorespaces#1}\par\vskip1ex
		\noindent\small{\itshape E-mail address}%
		\/: #2\par\vskip4ex}\hfill%
	\endgroup}%
\makeatother
\title{\sc Non-degenerate anisocurved surfaces in homogeneous 3-manifolds} %title of the paper
\author{\sc Alma L. Albujer and F\'abio R. dos Santos
	\textsc{} %names of authors
}
\date{\today} %leave empty
\begin{document}

\maketitle

%%%%%%%%%%%%%%% footnote %%%%%%%%%%%%%%%%
\abstract{In this manuscript we consider non-degenerate surfaces $\Sigma^2$ immersed in a $3$-dimensional homogeneous space $\mathbb{L}^{3}(\kappa,\tau)$ endowed with two different metrics, the one induced by the Riemannian metric of $\mathbb{E}^3(\kappa,\tau)$ and the non-degenerate metric inherited by the Lorentzian one of $\mathbb{L}^3(\kappa,\tau)$. Therefore, we have two different geometries on $\Sigma^2$ and we can compare them. In particular, we can consider the Gaussian curvature functions which respect to both metrics and study the geometry of the surfaces satisfying that both Gaussian curvature functions are opposite. We will call these surfaces anisocurved surfaces. In order to obtain our main results we also need to impose some extra assumptions regarding the extrinsic curvatures with respect to both metrics.}

\bigskip

\noindent %2020 MSC numbers

\smallskip

\noindent %keywords
\section{Introduction}
Kobayashi~\cite{Kobayashi:83} showed in 1983 that spacelike surfaces in the 3-dimensional Lorentz-Minkowski space $\mathbb{L}^3$ which are simultaneously minimal and maximal surfaces are necessarily open pieces of spacelike planes or of the helicoid, in the region where it is spacelike. Let us recall at this point that a spacelike surface in a Lorentzian manifold is an immersed surface such that the metric induced from the ambient space is a Riemannian one. Furthermore, a maximal surface in a Lorentzian manifold is a spacelike surface with identically zero mean curvature, whereas a minimal surface is a surface with zero mean curvature in a Riemannian manifold. Since any spacelike surface in $\mathbb{L}^3$ can also be endowed with a second Riemannian metric, the one induced from the Euclidean space $\mathbb{R}^3$, the problem studied by Kobayashi makes sense.

During the last years several authors have considered different extensions of the above result, generalizing it either to general dimension or to surfaces in different ambient spaces. Specifically, Kim, Koh, Shin and Yang~\cite{Kim:09} studied such surfaces in a Lorentzian product space of the type $M^2\times\mathbb{R}_1$, where $M^2$ is a Riemannian surface, and in particular they obtained a full classification in the case of $\mathbb{S}^2\times\mathbb{R}_1$ and $\mathbb{H}^2\times\mathbb{R}_1$. Furthermore, Shin, Kim, Koh, Lee and Yang~\cite{Shin:13} considered a similar problem in the Heisenberg space $\textrm{Nil}^3$, whose construction also admits a Lorentzian counterpart, classifying not only simultaneously minimal and maximal spacelike surfaces in $\textrm{Nil}^3$, but also non-degenerate ones.

More recently, Alías, Alarc\'on and dos Santos~\cite{Eva:19} generalized the problem to the study of simultaneosuly minimal and maximal non-degenerate hypersurfaces in an $(n+1)$-dimensional Lorentzian product space $M^n\times\mathbb{R}_1$, $M^n$ being a Riemannian manifold. Furthermore, the authors also got a relation between the Gaussian curvatures of a non-degenerate surface in a product space $M^2\times\mathbb{R}_1$ with respect to both metrics, the one inherited by $M^n\times\mathbb{R}_1$, $K_L$, and the one induced by the Riemannian product $M^n\times\mathbb{R}$, $K_R$. As a consequence, they obtained a classification of non-degenerate surfaces in a product space $M^2(c)\times\mathbb{R}_1$ such that $K_R=K_L=c$, where $M^2(c)$ is the $2$-dimensional space form of Gaussian curvature $c$.

Going a step further, it is possible to consider the well-known family of Bianchi-Cartan-Vranceanu (BCV) spaces $\mathbb{E}^{3}(\kappa,\tau)$ for any real constants $\kappa$ and $\tau$. It is well known that $\mathbb{E}^{3}(\kappa,\tau)$, with $\kappa\neq 4\tau^2$, models all the simply connected 3-dimensional manifolds with isometry group of dimension $4$. Observe that, in particular, $\mathbb{E}^3(\kappa,0)$ models the product space $M^2(\kappa)\times\mathbb{R}$ and, given $\tau\neq 0$, $\mathbb{E}^3(0,\tau)$ is isometric to the Heisenberg space. Finally, in the case where $\kappa$ and $\tau$ do not vanish, $\mathbb{E}^{3}(\kappa,\tau)$ is isometric to a Berger sphere when $\kappa>0$ and to the universal covering of the special linear group when $\kappa<0$. Furthermore, the spaces $\mathbb{E}^{3}(\kappa,\tau)$ have their Lorentzian counterparts $\mathbb{L}^{3}(\kappa,\tau)$ (see for instance~\cite{Lee:11}), being also 3-dimensional Lorentzian manifolds with isometry group of dimension 4 whenever $\kappa\neq -4\tau^2$. Therefore, the spaces $\mathbb{E}^{3}(\kappa,\tau)$ and $\mathbb{L}^{3}(\kappa,\tau)$ are a natural generalization of most of the previously considered ambient spaces. Since both spaces represent the same topological manifold, it makes sense to consider the same topological surface endowed with two different metrics.  The main objective in this manuscript is to study and give some geometric properties of non-degenerate surfaces in $\mathbb{L}^3(\kappa,\tau)$ having opposite Gaussian curvatures when considering those surfaces as immersions into $\mathbb{E}^3(\kappa,\tau)$ and $\mathbb{L}^3(\kappa,\tau)$. We call such surfaces anisocurved surfaces. Let us observe that the Gaussian curvature functions of any non-degenerate surface in $\mathbb{L}^3$, when considered it as an immersion into $\mathbb{R}^3$ and $\mathbb{L}^3$, have always different sign (see for instance~\cite[Proposition 4.10]{Eva:19}), so it makes more sense to consider opposite values of the Gaussian curvature functions rather than imposing them to coincide.

The manuscrit is organized as follows. In Section~\ref{sec:ELkt} a description of both Riemannian and Lorentzian BCV spaces is given and we present some nice relations between their first fundamental forms and their Levi-Civita connections. Later on, in Section~\ref{Sec:2} we study, from a local point of view, the geometry of non-degenerate surfaces in $\mathbb{L}^3(\kappa,\tau)$. Since such surfaces can be endowed with two different metrics, we have in fact two different semi-Riemannian surfaces, so we can compare their extrinsic geometries. Specifically, when making an appropriate choice of the normal vector fields, we obtain some interesting relations involving their normal vector fields and their shape operators. As a consequence of such relations, a non-degenerate surface is an helix surface in $\mathbb{L}^{3}(\kappa,\tau)$ if and only if it is also an helix surface as a surface in $\mathbb{E}^{3}(\kappa,\tau)$, see Corollary~\ref{cor:constantangle}.

In Section~\ref{sec:curv} we continue with the comparison started in Section~\ref{Sec:2}, but now we focus on the Gaussian and extrinsic curvatures related to both metrics. Let us recall that the extrinsic curvature of a non-degenerate surface is defined as the determinant of its shape operator. In this direction, we develop a relation between the extrinsic curvatures of a non-degenerate surface in $\mathbb{L}^3(\kappa,\tau)$ and in $\mathbb{E}^3(\kappa,\tau)$.

Finally, in Section~\ref{sec:5} we present our main results regarding non-degenerate anisocurved surfaces in $\mathbb{L}^3(\kappa,\tau)$. Furthermore, in order to get our results we need to impose an additional assumption on the extrinsic curvatures. In the particular case of timelike surfaces, we get a characterization of Hopf surfaces, see Theorem~\ref{teo:5.timelike}, and of timelike anisocurved helix surfaces, see Proposition~\ref{prop:5.timelike}. In the case of spacelike surfaces we obtain in Theorem~\ref{teo:5.spacelike} a non-existence result.

\section{On $3$-dimensional homogeneous manifolds}
\label{sec:ELkt}

Let $\kappa$ and $\tau$ be real numbers, and consider the region $\mathcal{D}$ of the Euclidean space $\mathbb{R}^{3}$ given by
\begin{equation}\label{eq:1.1}
	\mathcal{D}=\left\{
	\begin{array}{ccc}
		\mathbb{R}^{3} & \mbox{if} & \kappa\geq0, \\
		\mathbb{D}(2/\sqrt{-\kappa})\times\mathbb{R} & \mbox{if} & \kappa<0,
	\end{array}
	\right.
\end{equation}
where $\mathbb{D}(r)$ is the disk in $\mathbb{R}^2$ of radius $r$ centered at the origin. The so-called {\it Bianchi-Cartan-Vranceanu space} (BCV-space), $\mathbb{E}^3(\kappa,\tau)$, is the Riemannian manifold obtained when endowing $\mathcal{D}$ with the homogeneous Riemannian metric
\begin{equation}\label{eq:1.2}
	\langle\,,\rangle_{R}=\lambda^{2}(dx^{2}+dy^{2})+\left(dz+\tau\lambda(ydx-xdy)\right)^{2},\quad\lambda=\dfrac{1}{1+\frac{\kappa}{4}(x^{2}+y^{2})}.
\end{equation}

The BCV-spaces with $\kappa\neq 4\tau^2$ are the only simply connected $3$-dimensional homogeneous spaces with $4$-dimensional isometry group. It is well-known that, according to the constants $\kappa$ and $\tau$, they are classified as follows,
\begin{itemize}
	\item in the case $\tau=0$, $\mathbb{E}^3(\kappa,0)$ is isometric to the Riemannian product space $M^{2}(\kappa)\times\mathbb{R}$, where $M^{2}(\kappa)$ is the $2$-dimensional space form of Gaussian curvature $\kappa$, i.e. the Euclidean sphere $\mathbb{S}^{2}(1/\sqrt{\kappa})$ when $\kappa>0$, the Euclidean plane $\mathbb{R}^2$ when $\kappa=0$ or the hyperbolic plane $\mathbb{H}^{2}(1/\sqrt{-\kappa})$ when $\kappa<0$.
	\item in the case $\tau\neq0$, $\mathbb{E}^{3}(\kappa,\tau)$ is isometric to the Berger sphere $\mathbb{S}^{3}_{b}(\kappa,\tau)$ when $\kappa>0$, to the Heisenberg space ${\rm Nil}^{3}(\tau)$ when $\kappa=0$ and to the universal cover of the special linear group, $\widetilde{Sl_2}(\mathbb{R})(\kappa,\tau)$, when $\kappa<0$.
\end{itemize}

Related to any BCV-space, we can consider a Riemannian submersion $\pi:\mathbb{E}^{3}(\kappa,\tau)\rightarrow M^{2}(\kappa)$ with totally geodesic fibers and bundle curvature $\tau$. Furthermore $\xi=\partial_z$ is a unit Killing vector field on $\mathbb{E}^{3}(\kappa,\tau)$, which is vertical to the submersion $\pi$.

Otherwise, if we endow $\mathcal{D}$ with the Lorentzian metric
\begin{equation}\label{eq:1.9}
	\langle\,,\rangle_{L}=\lambda^{2}(dx^{2}+dy^{2})-\left(dz+\tau\lambda(ydx-xdy)\right)^{2},\quad\lambda=\dfrac{1}{1+\frac{\kappa}{4}(x^{2}+y^{2})},
\end{equation}
we obtain the so-called {\it Lorentz-Bianchi-Cartan-Vranceanu space} ($LBCV$-space), which has been denoted in the literature by $\mathbb{L}^{3}(\kappa,\tau)$ (see for instance~\cite{Lee:11,Lee:13}). In an analogous way to the Riemannian situation, it holds that $\xi=\partial_z$ is a timelike unit Killing vector field on $\mathbb{L}^{3}(\kappa,\tau)$, which is tangent to the fibers of the submersion $\pi$.

Let us assume now that $\tau\neq 0$. In this case, it is a standard fact (see for instance~\cite{Daniel:07}) that the following canonical frame $\{E_{1},E_{2},E_{3}\}$ on $\mathfrak{X}(\mathcal{D})$, defined as
\begin{equation}\label{eq:1.3}
	\left.
	\begin{array}{ccl}
		E_{1}&=&\lambda^{-1}\left(\cos(\sigma z)\partial_{x}+\sin(\sigma z)\partial_{y}\right)+\tau\left(x\sin(\sigma z)-y\cos(\sigma z)\right)\partial_{z},\\
		E_{2}&=&\lambda^{-1}\left(-\sin(\sigma z)\partial_{x}+\cos(\sigma z)\partial_{y}\right)+\tau\left(x\cos(\sigma z)+y\sin(\sigma z)\right)\partial_{z},\\
		E_{3}&=&\partial_{z},
	\end{array}
	\right.
\end{equation}
where $\sigma=\kappa/2\tau$, is an orthonormal frame with respect to $\langle\,,\rangle_R$, that is,
\begin{equation}\label{eq:1.4}
	\langle E_{1},E_{1}\rangle_{R}=\langle E_{2},E_{2}\rangle_{R}=\langle E_{3},E_{3}\rangle_{R}=1\quad\mbox{and}\quad\langle E_{i},E_{j}\rangle_{R}=0\quad\mbox{for}\quad i\neq j.
\end{equation}
Furthermore, it is easy to show that the frame $\{E_1,E_2,E_3\}$ defined as in~\eqref{eq:1.3} is also orthonormal in the corresponding $LBCV$-space, that is,
\begin{equation}\label{eq:1.10}
	\langle E_{1},E_{1}\rangle_{L}=\langle E_{2},E_{2}\rangle_{L}=1,\quad\langle E_{3},E_{3}\rangle_{L}=-1\quad\mbox{and}\quad\langle E_{i},E_{j}\rangle_{L}=0\quad\mbox{for}\quad i\neq j.
\end{equation}

Let us observe that, since $E_3=\xi$ is vertical to the submersion, $E_1$ and $E_2$ are horizontal vector fields both on $\mathbb{E}^{3}(\kappa,\tau)$ and on $\mathbb{L}^3(\kappa,\tau)$. Therefore, given $X\in\mathfrak{X}(\mathcal{D})$, we can write $X=X^{h}+X^{v}$, where $X^{h}\,\in\,\textrm{span}\{E_1,E_2\}$ and $X^{v}=\alpha E_3$ denote, respectively, the horizontal and vertical components of a vector field $X\in\mathfrak{X}(\mathcal{D})$. It follows immediately that, for all $X,Y\in\mathfrak{X}(\mathcal{D})$,
\begin{equation}\label{eq:1.7.bis}
	\alpha =\langle X,E_{3}\rangle_{R}=-\langle X,E_{3}\rangle_{L}\quad\mbox{and}\quad\langle X^{h},Y^{h}\rangle_{R}=\langle X^{h},Y^{h}\rangle_{L}.
\end{equation}
This decomposition provides an interesting relation between the first fundamental forms of both spaces.
\begin{lemma}\label{lem:1.1}
	Given $X,Y\in\mathfrak{X}(\mathcal{D})$, it holds
	\begin{equation}\label{eq:1.18}
		\langle X,Y\rangle_{R}+\langle X,Y\rangle_{L}=2\langle X^{h},Y^{h}\rangle_{R}=2\langle X^{h},Y^{h}\rangle_{L}
	\end{equation}
	and
	\begin{equation}\label{eq:1.19}
		\langle X,Y\rangle_{R}-\langle X,Y\rangle_{L}=2\langle X, E_{3}\rangle_{R}\langle Y, E_{3}\rangle_{R}=2\langle X, E_{3}\rangle_{L}\langle Y, E_{3}\rangle_{L}.
	\end{equation}
\end{lemma}

\begin{proof}
	The proof follows immediately from~\eqref{eq:1.7.bis} and from the already mentioned fact that the canonical frame on $\mathfrak{X}(\mathcal{D})$, $\{E_1,E_2,E_3\}$, is orthonormal with respect to both metrics.
\end{proof}

With a straigthforward computation, we can check that
\begin{equation}\label{eq:1.5}
	[E_{1},E_{2}]=2\tau E_{3},\quad[E_{2},E_{3}]=\sigma E_{1}\quad\mbox{and}\quad[E_{3},E_{1}]=\sigma E_{2}.
\end{equation}
Furthermore, since $\xi$ is a Killing vector field on $\mathbb{E}^3(\kappa,\tau)$, it follows easily that for any vector field $X\in\mathfrak{X}(\mathcal{D})$ the following identity holds
\begin{equation}\label{eq:1.7}
	\overline{\nabla}^{R}_{X}\xi=\tau(X\wedge_{R}\xi),
\end{equation}
where $\overline{\nabla}^R$ stands for the Levi-Civita connection in $\mathbb{E}^{3}(\kappa,\tau)$ and $\wedge_R$ is defined by $\langle X\wedge_R Y,Z\rangle_R=\mathrm{det}(X,Y,Z)$ for any vector fields $X,Y,Z\in\mathfrak{X}(\mathcal{D})$, (see~\cite{Daniel:07}). From~\eqref{eq:1.5},~\eqref{eq:1.7} and Koszul's formula, the following expressions can be derived
\begin{equation}\label{eq:1.8}
	\left.
	\begin{array}{cclccclcccl}
		\overline{\nabla}^{R}_{E_{1}}E_{1}&=&0,  &\quad & \overline{\nabla}^{R}_{E_{1}}E_{2}&=&\tau E_{3}, &\quad & \overline{\nabla}^{R}_{E_{1}}E_{3}&=&-\tau E_{2},\\
		\overline{\nabla}^{R}_{E_{2}}E_{1}&=&-\tau E_{3},  &\quad & \overline{\nabla}^{R}_{E_{2}}E_{2}&=&0, &\quad & \overline{\nabla}^{R}_{E_{2}}E_{3}&=&\tau E_{1},\\
		\overline{\nabla}^{R}_{E_{3}}E_{1}&=&\dfrac{\kappa-2\tau^{2}}{2\tau}E_{2},  &\quad & \overline{\nabla}^{R}_{E_{3}}E_{2}&=&-\dfrac{\kappa-2\tau^{2}}{2\tau}E_{1}, &\quad & \overline{\nabla}^{R}_{E_{3}}E_{3}&=&0.
	\end{array}
	\right.
\end{equation}
Let us refer the reader to~\cite{Daniel:07} and~\cite{Espinar:11} for a deeper study of the geometry of the $BCV$-spaces.

Similarly, in the case of an $LBCV$-space $\mathbb{L}^3(\kappa,\tau)$, the timelike Killing vector field $\xi$ satisfies
\begin{equation}\label{eq:1.11}
	\overline{\nabla}^{L}_{X}\xi=-\tau(X\wedge_{L}\xi),
\end{equation} where
$\overline{\nabla}^{L}$ denotes the Levi-Civita connection in $\mathbb{L}^{3}(\kappa,\tau)$ and $\wedge_{L}$ is given by $\langle X \wedge_L Y,Z\rangle_L=\textrm{det}(X,Y,Z)$ for any $X,Y,Z\in\mathfrak{X}(\mathcal{D})$. In an analogous way as in the Riemannian case, from ~\eqref{eq:1.5},~\eqref{eq:1.11} and Koszul's formula we also get
\begin{equation}\label{eq:1.13}
	\left.
	\begin{array}{cclccclcccl}
		\overline{\nabla}^{L}_{E_{1}}E_{1}&=&0,  &\quad & \overline{\nabla}^{L}_{E_{1}}E_{2}&=&\tau E_{3}, &\quad & \overline{\nabla}^{L}_{E_{1}}E_{3}&=&\tau E_{2},\\
		\overline{\nabla}^{L}_{E_{2}}E_{1}&=&-\tau E_{3},  &\quad & \overline{\nabla}^{L}_{E_{2}}E_{2}&=&0, &\quad & \overline{\nabla}^{L}_{E_{2}}E_{3}&=&-\tau E_{1},\\
		\overline{\nabla}^{L}_{E_{3}}E_{1}&=&\dfrac{\kappa+2\tau^{2}}{2\tau}E_{2},  &\quad & \overline{\nabla}^{L}_{E_{3}}E_{2}&=&-\dfrac{\kappa+2\tau^{2}}{2\tau}E_{1}, &\quad & \overline{\nabla}^{L}_{E_{3}}E_{3}&=&0.
	\end{array}
	\right.
\end{equation}

We end this comparative study of the spaces $\mathbb{E}^3(\kappa,\tau)$ and $\mathbb{L}^3(\kappa,\tau)$ by establishing the following relation between the Levi-Civita connections $\overline{\nabla}^R$ and $\overline{\nabla}^L$.

\begin{lemma}\label{lem:1.2}
	The Levi-Civita connections of the homogeneous spaces $\mathbb{E}^{3}(\kappa,\tau)$ and $\mathbb{L}^{3}(\kappa,\tau)$ are related by
	\begin{equation*}
		\overline{\nabla}^{R}_{Y}X-\overline{\nabla}^{L}_{Y}X=W(X,Y),
	\end{equation*}
	where
	\begin{equation}\label{eq:1.22}
		W(X,Y)=2\tau\left(X^{h}\wedge_{R}Y^{v}-X^{v}\wedge_{R}Y^{h}\right)=2\tau\left(X^{h}\wedge_{L}Y^{v}-X^{v}\wedge_{L}Y^{h}\right)
	\end{equation}
	for all $X,Y\in\mathfrak{X}(\mathcal{D})$.
\end{lemma}

\begin{proof}
	Recalling that any $X\in\mathfrak{X}(\mathcal{D})$ admits the splitting $X=X^h+X^v$, where $X^h\in\mathrm{span}(E_1,E_2)$ and $X^v=\alpha E_3$, it follows easily from~\eqref{eq:1.8} and~\eqref{eq:1.13} that
	\begin{equation*}
		\overline{\nabla}^{R}_{Y^h}X^h-\overline{\nabla}^{L}_{Y^h}X^h=0 \quad \textrm{and} \quad \overline{\nabla}^{R}_{Y^v}X^v-\overline{\nabla}^{L}_{Y^v}X^v=0,
	\end{equation*}
	and consequently
	\begin{equation}\label{eq:1.23}
		\overline{\nabla}^{R}_{Y}X-\overline{\nabla}^{L}_{Y}X= \left(\overline{\nabla}^{R}_{Y^h}X^v-\overline{\nabla}^{L}_{Y^h}X^v\right)+\left(\overline{\nabla}^{R}_{Y^v}X^h-\overline{\nabla}^{L}_{Y^v}X^h\right).
	\end{equation}
	
	Let us now consider the first term in the right hand side of~\eqref{eq:1.23}. Using again~\eqref{eq:1.8} and~\eqref{eq:1.13} we get
	\begin{equation}\label{eq:1.24}
		\overline{\nabla}^{R}_{Y^h}X^v-\overline{\nabla}^{L}_{Y^h}X^v=2\tau\langle X^v,E_3\rangle_R\left(\langle Y^h,E_2\rangle_R E_1-\langle Y^h,E_1\rangle_R E_2\right).
	\end{equation}
	Furthermore, from the definition of $\wedge_R$ it yields that  $E_3\wedge_R Y^h=\langle Y^h,E_1\rangle_R E_2-\langle Y^h,E_2\rangle_R E_1$. Therefore,
	\begin{equation}\label{eq:1.25}
		\overline{\nabla}^{R}_{Y^h}X^v-\overline{\nabla}^{L}_{Y^h}X^v=-2\tau(X^v\wedge_RY^h).
	\end{equation}
	Analogously, considering the last term in~\eqref{eq:1.23} we obtain
	\begin{equation}\label{eq:1.26}
		\overline{\nabla}^{R}_{Y^v}X^h-\overline{\nabla}^{L}_{Y^v}X^h=2\tau(X^h\wedge_R Y^v).
	\end{equation}
	
	Finally, taking into account~\eqref{eq:1.25} and~\eqref{eq:1.26}, equation~\eqref{eq:1.23} reads
	\begin{equation}\label{eq:1.27}
		\overline{\nabla}^{R}_{Y}X-\overline{\nabla}^{L}_{Y}X=2\tau\left(X^h\wedge_R Y^v-X^v\wedge_R Y^h\right),
	\end{equation}
	and the proof follows by observing that $X^h\wedge_RY^v=X^h\wedge_LY^v$ for any $X,Y\in\mathfrak{X}(\mathcal{D})$.
\end{proof}

\begin{remark}
	Observe that in the case $\tau=0$, it is possible to reproduce the above study by considering in the product spaces $\mathbb{E}^3(\kappa,0)$ and $\mathbb{L}^3(\kappa,0)$ the orthonormal frame $\{\partial_x,\partial_y,\partial_z\}$ on $\mathfrak{X}(\mathcal{D})$ instead of $\{E_1,E_2,E_3\}$. In particular, $\partial_x$ and $\partial_y$ are horizontal vectors and one can also consider the vectorial products $\wedge_R$ and $\wedge_L$, giving their expressions in terms of $\{\partial_x,\partial_y,\partial_z\}$. This case was previously studied in~\cite{Eva:19}. Specifically, Lemma~\ref{lem:1.1} was obtained in~\cite{Eva:19} for $\tau=0$, and it was proved that, in this situation, the Riemannian and Lorentzian Levi-Civita connections coincide, so Lemma~\ref{lem:1.2} also holds.
\end{remark}

\section{Surfaces in homogeneous spaces}\label{Sec:2}

Let us recall that a smooth immersion $\psi:\Sigma^2\rightarrow\mathbb{L}^3(\kappa,\tau)$ of a connected surface $\Sigma^2$ is said to be a \emph{non-degenerate} surface if $\psi$ induces, from the Lorentzian metric of $\mathbb{L}^3(\kappa,\tau)$, a non-degenerate metric on $\Sigma^2$ which, as usual, is also denoted by $\langle,\rangle_L$. In this case there are just two possibilities, either the metric induced on $\Sigma^2$ is a Riemannian one, and in this case we say that $\Sigma^2$ is a \emph{spacelike} surface, or it is Lorentzian, and $\Sigma^2$ is said to be a \emph{timelike} surface. Equivalently, $\Sigma^2$ is a spacelike (timelike, resp.) surface if and only if for every $p\in\Sigma^2$ the tangent plane $\psi_\ast(T_p\Sigma)$ is a spacelike (timelike, resp.) plane.

In the case when $\Sigma^{2}$ is spacelike, since $\xi$ is a unit timelike vector field globally defined on $\mathbb{L}^{3}(\kappa,\tau)$, there exists a unique global timelike unit normal vector field on $\Sigma^{2}$, $N_{L}$, which is in the same time-orientation as $\xi$, and hence we may assume that $\Sigma^{2}$ is oriented by $N_{L}$. As a consequence of the backwards Schwarz inequality in the Lorentzian context, it follows that
\begin{equation}\label{eq:2.1}
	\langle N_{L},\xi\rangle_{L}\leq-1<0,
\end{equation}
with equality if and only if $N_{L}=\xi$. In particular, there exists a unique number $\varphi\geq0$, called the {\em hyperbolic angle} between $N_{L}$ and $\xi$, such that
\begin{equation}\label{eq:2.2}
	\langle N_{L},\xi\rangle_{L}=-\cosh\varphi.
\end{equation}
Nevertheless, in the case when $\Sigma^2$ is timelike we cannot assure the global existence of a normal vector field $N_L$ satisfying the inequality~\eqref{eq:2.1}. However, such vector field can be locally defined in almost all $\Sigma^2$. Specifically, given a non-degenerate surface $\Sigma^2$ in $\mathbb{L}^{3}(\kappa,\tau)$, we can define the open region of $\Sigma^2$ given by
\begin{equation}\label{eq:IWOK}
	\widehat{\Sigma}^2=\{p\in\Sigma^2\,:\,\exists\, \mathcal{U}_p,\, N_L \textrm{ s.t. } \langle N_L,\xi\rangle_L \leq 0 \textrm{ on } \mathcal{U}_p\},
\end{equation}
where $\mathcal{U}_p\subseteq \Sigma^2$ is a neighbourhood of $p$ and $N_L$ is a unit normal vector field (locally) defined on $\widehat{\Sigma}^2$. In the case of a spacelike surface it trivially holds that $\widehat{\Sigma}^2=\Sigma^2$, whereas in the case of a timelike surface $\widehat{\Sigma}^2$ is an open and dense subset of $\Sigma^2$, and we can globally defined $N_L$ saystifying~\eqref{eq:2.1} in any connected component of $\widehat{\Sigma}^2$. In order to unify our notation, we will denote by $\varepsilon$ the sign of $\langle N_{L},N_{L}\rangle_{L}$, that is,
\begin{equation}\label{eq:2.4}
	\varepsilon=\left\{
	\begin{array}{rl}
		-1 & \mbox{if}\,\, \Sigma^{2}\,\, \mbox{is spacelike,}  \\
		1 & \mbox{if}\,\, \Sigma^{2}\,\, \mbox{is timelike.}
	\end{array}
	\right.
\end{equation}

Given $\psi:\Sigma^{2}\longrightarrow \mathbb{L}^{3}(\kappa,\tau)$ a non-degenerate surface, $\psi:\Sigma^{2}\longrightarrow\mathbb{E}^{3}(\kappa,\tau)$ also defines an immersed surface into the Riemannian homogeneous space $\mathbb{E}^{3}(\kappa,\tau)$. Let us denote the induced metric from $\mathbb{E}^3(\kappa,\tau)$ by $\langle,\rangle_R$. Let us see how it is possible to relate the extrinsic geometries of both surfaces. Firstly, we can get the following relation between both normal vector fields.

\begin{proposition}\label{prop:2.1}
	Let $\psi:\Sigma^{2}\longrightarrow \mathbb{L}^{3}(\kappa,\tau)$ be a non-degenerate surface immersed into the Lorentzian homogeneous space $\mathbb{L}^{3}(\kappa,\tau)$, and consider its restriction to any connected component of $\widehat{\Sigma}^2$, $\widehat{\Sigma}^2_0$. Let $N_L$ be the globally defined unit normal vector field on $\widehat{\Sigma}^2_0$ such that $\langle N_{L},\xi\rangle_{L}\leq0$. Then,
	\begin{equation}\label{normal R}
		N_{R}=\dfrac{1}{\omega_{L}}\left(\sqrt{2(\omega_{L}^{2}-\varepsilon)}\,\xi-N_{L}\right)
	\end{equation}
	globally defines the unique unit normal vector field on  $(\widehat{\Sigma}^2_0,\langle,\rangle_R)$ such that $\langle N_{R},\xi\rangle_{R}\geq0$, where
	\begin{equation}\label{eq:2.5}
		\omega_{L}=\sqrt{\varepsilon+2\langle N_{L},\xi\rangle_{L}^{2}}\geq1>0.
	\end{equation}
\end{proposition}

\begin{proof}
	First of all, from~\eqref{normal R} we get
	\begin{equation}\label{eq:2.6}
		\langle N_{R},N_{R}\rangle_{R}=\dfrac{1}{\omega_{L}^{2}}\left(2(\omega_{L}^{2}-\varepsilon)-2\sqrt{2(\omega_{L}^{2}-\varepsilon)}\langle N_{L},\xi\rangle_{R}+\langle N_{L},N_{L}\rangle_{R}\right).
	\end{equation}
	Since $\langle N_{L},\xi\rangle_{L}\leq0$, it follows from Lemma~\ref{lem:1.1} and~\eqref{eq:2.5} that
	\begin{equation}\label{eq:2.7}
		\langle N_L,\xi\rangle_{R}=-\langle N_{L},\xi\rangle_{L}=\dfrac{\sqrt{2(\omega_{L}^{2}-\varepsilon)}}{2},
	\end{equation}
	and
	\begin{equation}\label{eq:2.13}
		\langle N_{L},N_{L}\rangle_{R}=\varepsilon+2\langle N_{L},\xi\rangle_{L}^{2}=\omega_{L}^{2}.
	\end{equation}
	Thus, we can immediately check from~\eqref{eq:2.6} that $\langle N_{R},N_{R}\rangle_{R}=1$.
	
	Furthermore, from~\eqref{normal R} we also obtain
	\begin{equation}\label{eq:2.15.bis}
		\langle N_{R},\xi\rangle_{R}=\dfrac{1}{\omega_{L}}\left(\sqrt{2(\omega_{L}^{2}-\varepsilon)}-\dfrac{\sqrt{2(\omega_{L}^{2}-\varepsilon)}}{2}\right)=\dfrac{\sqrt{2(\omega_{L}^{2}-\varepsilon)}}{2\omega_{L}}\geq0.
	\end{equation}
	Finally, using again~\eqref{normal R},~\eqref{eq:2.7} and Lemma~\ref{lem:1.1}, for every $X\in\mathfrak{X}(\Sigma)$ we obtain
	\begin{align}\label{eq:2.16}
		\langle N_{R},X\rangle_{R}=&\dfrac{1}{\omega_{L}}\left(\sqrt{2(\omega^{2}_{L}-\varepsilon)}\langle \xi,X\rangle_{R}-\langle N_{L},X\rangle_{R}\right)\\
		=&\dfrac{1}{\omega_{L}}\left(-\sqrt{2(\omega^{2}_{L}-\varepsilon)}-2\langle N_{L},\xi\rangle_{L}\right)\langle\xi, X\rangle_L=0,
	\end{align}
	which finishes the proof.
\end{proof}

Observe that it is also possible to express the vector field $N_L$ in terms of $N_R$. In fact, from~\eqref{normal R} we immediately get
\begin{equation}\label{normal L}
	N_L=\frac{1}{\omega_R}\left(\sqrt{2(1-\varepsilon \omega^2_R)}\,\xi - N_R\right),
\end{equation}
where, by definition, $\omega_R:=\dfrac{1}{\omega_L}$. Moreover, from~\eqref{eq:2.15.bis} it holds
\begin{equation}\label{omega_R}
	\omega_R=\sqrt{\varepsilon (1-2\langle N_R,\xi\rangle^2_R)}.
\end{equation}

In what follows, we will always assume that given a non-degenerate surface $\Sigma^{2}$ in $\mathbb{L}^{3}(\kappa,\tau)$, any connected component $\widehat{\Sigma}^2_0$ of $\widehat{\Sigma}$ is oriented by the unique unit normal vector field $N_{L}$ such that $\langle N_{L},\xi\rangle_{L}\leq0$. Doing so, as a surface of the space $\mathbb{E}^{3}(\kappa,\tau)$, $\widehat\Sigma^{2}_0$ will be always oriented by the unit normal vector field $N_{R}$ given by~\eqref{normal R}, with $\langle N_{R},\xi\rangle_{R}\geq0$. We observe that, given any $V\in\mathfrak{X}(D)$,~\eqref{normal R} yields
\begin{equation}\label{eq:2.24_aux}
	\begin{split}
		\langle N_{R},V\rangle_{R}=&\dfrac{1}{\omega_{L}}\sqrt{2(\omega_{L}^{2}-\varepsilon)}\langle \xi,V\rangle_{R}-\dfrac{1}{\omega_{L}}\langle N_{L},V\rangle_{R}\\
		=&-\dfrac{1}{\omega_{L}}\sqrt{2(\omega_{L}^{2}-\varepsilon)}\langle  \xi,V\rangle_{L}-\dfrac{1}{\omega_{L}}\langle N_{L},V\rangle_{R}.
	\end{split}
\end{equation}
Furthermore, from~\eqref{eq:1.19} and~\eqref{eq:2.7} it holds
\begin{equation}\label{eq:2.24_aux2}
	\langle N_{L},V\rangle_{R}=\langle N_{L},V\rangle_{L}+2\langle N_{L},\xi\rangle_{L}\langle \xi,V\rangle_{L}=\langle N_{L},V\rangle_{L}-\sqrt{2(\omega_{L}^{2}-\varepsilon)}\langle \xi,V\rangle_{L}.
\end{equation}
Hence,
\begin{equation}\label{eq:2.24}
	\langle N_{R},V\rangle_{R}=-\dfrac{1}{\omega_{L}}\langle N_{L},V\rangle_{L} \quad \textrm{on} \quad \widehat{\Sigma}^2_0.
\end{equation}

An orientable surface $\Sigma^{2}$ in $\mathbb{E}^{3}(\kappa,\tau)$ and/or $\mathbb{L}^{3}(\kappa,\tau)$ is said to be an {\em helix surface}, or to be a {\em constant angle surface}, if its normal vector field makes a constant angle with respect to the vertical vector field $\xi$. As a particular case of helix surface, $\Sigma^2$ is called a \emph{Hopf surface} in $\mathbb{E}^3(\kappa,\tau)$ when $\langle N_R,\xi\rangle_R=0$. It is well-known that Hopf surfaces in $\mathbb{E}^3(\kappa,\tau)$ are the preimage by the submersion $\pi$ of a regular curve $\alpha$ in $M^2(\kappa)$, $\pi^{-1}(\alpha)$, and they are always timelike surfaces in $\mathbb{L}^3(\kappa,\tau)$. On the other hand, in the case $\tau=0$, slices in $M^2(\kappa)\times \mathbb{R}$ are characterized by $N_R=N_L=\xi$, or equivalently by $\langle N_R,\xi\rangle_R=1$, and they are trivially spacelike surfaces in $M^2(\kappa)\times\mathbb{R}_1$. However, there is no surface in $\mathbb{E}^3(\kappa,\tau)$, $\tau\neq 0$, such that $N_R=\xi$, since in this case the horizontal distribution spanned by the vector fields $E_1$ and $E_2$ is not integrable.

It is worth pointing out that a full classification of helix surfaces in all the Riemannian product spaces $\mathbb{E}(\kappa,\tau)$ is known, see~\cite{Dillen:07,Dillen:09,diScala:09,Fastenakels:11,Onnis:14,Onnis:16,Munteanu:09}.

Observe that taking $V=\xi$ in~\eqref{eq:2.24}, we get the following nice consequence.

\begin{corollary}\label{cor:constantangle}
	A non-degenerate surface is an helix surface in $\mathbb{L}^{3}(\kappa,\tau)$ if and only if it is also an helix surface as a surface in $\mathbb{E}^{3}(\kappa,\tau)$.
\end{corollary}
\begin{proof}
	The corollary holds immediately in any connected component of $\widehat{\Sigma}^2$ from~\eqref{eq:2.24} and~\eqref{eq:2.5}. Since $\widehat{\Sigma}^2$ is dense on the (connected) surface $\Sigma^2$, the result follows by a continuity argument.
\end{proof}

Our next aim is to obtain a relation between the shape operators related to $(\Sigma^2,\langle,\rangle_R)$ and  $(\Sigma^2,\langle,\rangle_L)$. Before that, it is necessary to recall the integrability equations of such surfaces derived from the Gauss and Weingarten formulas.

From now on, and unless otherwise stated, we will assume that we are working on a connected component $\widehat{\Sigma}^2_0$ of $\widehat{\Sigma}^2$. Furthermore, according to Proposition~\ref{prop:2.1}, we will consider the normal vector fields $N_R$ and $N_L$ on $(\widehat{\Sigma}^2_0,\langle,\rangle_R)$ and $(\widehat{\Sigma}^2_0,\langle,\rangle_L)$ respectively, such that $\langle N_R,\xi\rangle_R\geq 0$ and $\langle N_L,\xi\rangle_L\leq 0$. Let us denote by $T_{R}$ and $T_{L}$ the tangential components of $\xi$ along $\widehat\Sigma^{2}_0$ with respect to the metrics $\langle\,,\rangle_{R}$ and $\langle\,,\rangle_{L}$, respectively. In this setting, we can consider the following splittings
\begin{equation}\label{eq:2.25}
	\xi=T_{R}+\langle N_{R},\xi\rangle_{R}N_{R}\quad\mbox{and}\quad\xi=T_{L}+\varepsilon\langle N_{L},\xi\rangle_{L}N_{L}.
\end{equation}
Thus, taking norms in~\eqref{eq:2.25} the followings identities hold,
\begin{equation}\label{eq:2.26}
	1=\langle\xi,\xi\rangle_{R}=|T_{R}|_{R}^{2}+\langle N_{R},\xi\rangle_{R}^{2}\quad\mbox{and}\quad-1=\langle\xi,\xi\rangle_{L}=|T_{L}|^2_{L}+\varepsilon\langle N_{L},\xi\rangle_{L}^{2},
\end{equation}
where $|\cdot|_R=\sqrt{\langle\cdot,\cdot\rangle_R}$ and $|\cdot|_L=\sqrt{|\langle\cdot,\cdot\rangle_{L}|}$ denote the norm on $(\Sigma^2,\langle,\rangle_R)$ and $(\Sigma^2,\langle,\rangle_L)$, respectively.

Let us denote by $\nabla^R$ the Levi-Civita connection related to $(\Sigma^2,\langle,\rangle_R)$. Then, the Gauss and Weingarten formulas of the surface $\psi:\Sigma^{2}\longrightarrow\mathbb{E}^{3}(\kappa,\tau)$ are given, respectively, by
\begin{equation}\label{eq:3.1}
	\overline{\nabla}^{R}_{X}Y=\nabla_{X}^{R}Y+\langle A_{R}(X),Y\rangle_{R}N_{R}
\end{equation}
and
\begin{equation}\label{eq:3.2}
	A_{R}(X)=-\overline{\nabla}^{R}_{X}N_{R},
\end{equation}
for every tangent vector fields $X,Y\in\mathfrak{X}(\Sigma)$, where $A_{R}:\mathfrak{X}(\Sigma)\rightarrow\mathfrak{X}(\Sigma)$ stands for the shape operator of $(\Sigma^{2},\langle,\rangle_R)$ with respect to $N_{R}$. From the above Gauss and Weingarten formulas, jointly with~\eqref{eq:1.7} and~\eqref{eq:2.25}, we obtain
\begin{align}\label{eq:3.8}
	\tau(X\wedge_{R}\xi)=&\overline{\nabla}^{R}_{X}\xi=\overline{\nabla}^{R}_{X}\left(T_{R}+\langle N_{R},\xi\rangle_{R}N_{R}\right)\\
	=&\nabla^{R}_{X} T_{R}+\langle A_{R}(X), T_{R}\rangle_{R}N_{R}+ X\left(\langle N_{R},\xi\rangle_{R}\right)N_{R}-\langle N_{R},\xi\rangle_{R}A_{R}(X).
\end{align}
On the other hand,
\begin{align}\label{eq:3.9}
	\tau(X\wedge_{R}\xi)=&\tau\left(X\wedge_{R}\left(T_{R}+\langle N_{R},\xi\rangle_{R}N_{R}\right)\right)\\
	=&-\tau\langle N_{R},\xi\rangle_{R}J_{R}(X)-\tau\langle J_{R}(T_{R}),X\rangle_{R}N_{R},
\end{align}
where $J_{R}(X)=N_{R}\wedge_{R}X$ for all $X\in\mathfrak{X}(\Sigma)$, so that $X\wedge_{R} T_{R}=-\langle X,J_{R}(T_{R})\rangle_{R}N_{R}$. Therefore, comparing the tangent and normal components in~\eqref{eq:3.8} and~\eqref{eq:3.9}, we can derive the following integrability equations,
\begin{equation}\label{eq:3.10}
	\nabla^{R}_{X}T_{R}=\langle N_{R},\xi\rangle_{R}\left(A_{R}(X)-\tau J_{R}(X)\right)
\end{equation}
and
\begin{equation}\label{eq:3.11}
	X(\langle N_{R},\xi\rangle_{R})=-\langle(A_{R}+\tau J_{R})(T_{R}),X\rangle_{R},
\end{equation}
for all $X\in\mathfrak{X}(\Sigma)$.

Analogously, the Gauss and Weingarten formulas of the non-degenerate surface $\psi:\Sigma^{2}\longrightarrow\mathbb{L}^{3}(\kappa,\tau)$ are given, respectively, by
\begin{equation}\label{eq:3.3}
	\overline{\nabla}^{L}_{X}Y=\nabla_{X}^{L}Y+\varepsilon\langle A_{L}(X),Y\rangle_{L}N_{L}
\end{equation}
and
\begin{equation}\label{eq:WeingR}
	A_{L}(X)=-\overline{\nabla}^{L}_{X}N_{L},
\end{equation}
for every tangent vector fields $X,Y\in\mathfrak{X}(\Sigma)$, where $\nabla^L$ and $A_{L}:\mathfrak{X}(\Sigma)\rightarrow\mathfrak{X}(\Sigma)$ stand for the Levi-Civita connection and the shape operator of $(\Sigma^{2},\langle,\rangle_L)$ with respect to $N_{L}$, respectively. Moreover, the corresponding integrability equations are given by
\begin{equation}\label{eq:int1_L}
	\nabla^{L}_{X} T_{L}=\varepsilon\langle N_{L},\xi\rangle_{L}\left(A_{L}(X)+\tau J_{L}(X)\right)
\end{equation}
and
\begin{equation}\label{eq:int2_L}
	X\left(\langle N_{L},\xi\rangle_{L}\right)=-\langle(A_{L}-\tau J_{L})(T_{L}),X\rangle_L,
\end{equation}
for all $X\in\mathfrak{X}(\Sigma)$, where $J_L(X)=N_L\wedge_L X$.

In our next result we establish a relation between both shape operators $A_R$ and $A_L$.
\begin{proposition}\label{prop:3.1}
	Let $\Sigma^{2}$ be a non-degenerate surface in $\mathbb{L}^{3}(\kappa,\tau)$ and let $\widehat{\Sigma}^2_0$ be a connected component of $\widehat{\Sigma}^2$. Then, the corresponding shape operators of $(\Sigma^2,\langle,\rangle_R)$ and $(\Sigma^2,\langle,\rangle_L)$ with respect to $N_R$ and $N_L$ in $\widehat{\Sigma}^2_0$ are related by
	\begin{equation}\label{shape R}
		A_{R}(X)=-\dfrac{1}{\omega_{L}}A_{L}(X)-\dfrac{2\varepsilon}{\omega_{L}^{3}}\langle(A_{L}-\tau J_{L})(T_{L}),X\rangle_{L}T_{L}-\dfrac{2\tau}{\omega_{L}}\langle T_{L},X\rangle_{L}J_{L}(T_{L}),
	\end{equation}
	for every $X\in\mathfrak{X}(\widehat{\Sigma}_0)$. Equivalently,
	\begin{equation}\label{shape L}
		A_{L}(X)=-\dfrac{1}{\omega_{R}}A_{R}(X)+\dfrac{2\varepsilon}{\omega_{R}^{3}}\langle (A_{R}+\tau J_{R})(T_{R}),X\rangle_{R}T_{R}-\dfrac{2\tau}{\omega_{R}}\langle T_{R},X\rangle_{R}\,J_{R}(T_{R}).
	\end{equation}
\end{proposition}

\begin{proof}
	Let us prove relation~\eqref{shape R} since the second one is analogous. Lemma~\ref{lem:1.2}, jointly with expression~\eqref{normal R}, yields
	\begin{align}\label{eq:3.12}
		A_{R}(X)=&-\overline{\nabla}^{R}_{X}N_{R}=-\overline{\nabla}^{L}_{X}N_{R}-W(N_{R},X)\\
		=&-\overline{\nabla}^{L}_{X}\left(\dfrac{1}{\omega_{L}}\left(\sqrt{2(\omega^{2}_{L}-\varepsilon)}\, \xi-N_{L}\right)\right)-W(N_{R},X)\\
		 =&\dfrac{X(\omega_{L})}{\omega^{2}_{L}}\left(\sqrt{2(\omega^{2}_{L}-\varepsilon)}\,\xi-N_{L}\right)+\dfrac{1}{\omega_{L}}\overline{\nabla}^{L}_{X}N_{L}-\dfrac{1}{\omega_{L}}X\left(\sqrt{2(\omega^{2}_{L}-\varepsilon)}\right)\xi\\
		&-\,\dfrac{1}{\omega_{L}}\sqrt{2(\omega^{2}_{L}-\varepsilon)}\overline{\nabla}^{L}_{X}\xi-W(N_{R},X),
	\end{align}
	for every $X\in\mathfrak{X}(\widehat{\Sigma}_0)$. On the one hand, from~\eqref{eq:2.7} we have
	\begin{equation}\label{eq:3.13}
		X\left(\langle N_{L},\xi\rangle_{L}^{2}\right)=2\langle N_{L},\xi\rangle_{L}X\left(\langle N_{L},\xi\rangle_{L}\right)=-X\left(\langle N_{L},\xi\rangle_{L}\right)\sqrt{2(\omega^{2}_{L}-\varepsilon)},
	\end{equation}
	so that
	\begin{equation}\label{eq:3.14}
		X(\omega_{L})=X\left(\sqrt{\varepsilon+2\langle N_{L},\xi\rangle_{L}^{2}}\right)=-\dfrac{\sqrt{2(\omega^{2}_{L}-\varepsilon)}}{\omega_{L}}X\left(\langle N_{L},\xi\rangle_{L}\right).
	\end{equation}
	Consequently,
	\begin{equation}\label{eq:3.15}
		\dfrac{X(\omega_{L})}{\omega^{2}_{L}}=-\dfrac{\sqrt{2(\omega^{2}_{L}-\varepsilon)}}{\omega_{L}^{3}}X\left(\langle N_{L},\xi\rangle_{L}\right)\quad\mbox{and}\quad X\left(\sqrt{2(\omega^{2}_{L}-\varepsilon)}\right)=-2X\left(\langle N_{L},\xi\rangle_{L}\right).
	\end{equation}
	Hence, inserting~\eqref{eq:3.15} in~\eqref{eq:3.12} we obtain
	\begin{align}\label{eq:3.17}
		A_{R}(X)&=-\dfrac{1}{\omega_{L}}A_{L}(X)+\dfrac{1}{\omega_{L}^{3}}X\left(\langle N_{L},\xi\rangle_{L}\right)\left(2\varepsilon\,\xi+\sqrt{2(\omega_{L}^{2}-\varepsilon)}\,N_{L}\right)\\
		&-\dfrac{1}{\omega_{L}}\sqrt{2(\omega^{2}_{L}-\varepsilon)}\overline{\nabla}^{L}_{X}\xi-W(N_{R},X).
	\end{align}
	On the other hand, using one more time~\eqref{eq:2.7} and~\eqref{eq:2.25}, it holds
	\begin{equation}\label{eq:3.18}
		\xi=T_{L}-\dfrac{\varepsilon}{2}\sqrt{2(\omega_{L}^{2}-\varepsilon)}N_{L}.
	\end{equation}
	Therefore, inserting~\eqref{eq:3.18} in~\eqref{eq:3.17} we get
	\begin{equation}\label{eq:3.19}
		A_{R}(X)=-\dfrac{1}{\omega_{L}}A_{L}(X)+\dfrac{2\varepsilon}{\omega_{L}^{3}}X\left(\langle N_{L},\xi\rangle_{L}\right)T_{L}-\dfrac{1}{\omega_{L}}\sqrt{2(\omega^{2}_{L}-\varepsilon)}\overline{\nabla}^{L}_{X}\xi-W(N_{R},X).
	\end{equation}
	A direct computation from~\eqref{eq:1.22},~\eqref{normal R} and~\eqref{eq:2.7} implies
	\begin{equation}\label{eq:3.20}
		\begin{split}
			W(N_{R},X)&=-\dfrac{2}{\omega_{L}}\left(\langle X, \xi\rangle_{L}\overline{\nabla}^{L}_{N_{L}} \xi+\sqrt{2(\omega_{L}^{2}-\varepsilon)}\overline{\nabla}^{L}_{X}\xi+\langle N_{L},\xi\rangle_{L}\overline{\nabla}^{L}_{X}\xi\right)\\
			&=-\dfrac{2}{\omega_{L}}\left(\langle X, \xi\rangle_{L}\overline{\nabla}^{L}_{N_{L}} \xi+\dfrac{\sqrt{2(\omega_{L}^{2}-\varepsilon)}}{2}\overline{\nabla}^{L}_{X}\xi\right),
		\end{split}
	\end{equation}
	for all $X\in\mathfrak{X}(\Sigma)$. Furthermore, from~\eqref{eq:1.11} and~\eqref{eq:2.25} we have
	\begin{equation}\label{eq:3.22}
		\overline{\nabla}^{L}_{N_{L}} \xi=-\tau(N_{L}\wedge_{L} \xi)= -\tau J_{L}(T_{L}).
	\end{equation}
	The desired expression follows now easily from~\eqref{eq:int2_L},~\eqref{eq:3.19},~\eqref{eq:3.20} and~\eqref{eq:3.22}.
\end{proof}

\section{Gaussian and extrinsic curvature of surfaces in homogeneous spaces}\label{sec:curv}

Following~\cite{Daniel:07}, the curvature tensor\footnote{We adopt for the $(1,3)$-curvature tensor of a semi-Riemannian manifold the following definition (\cite[Chapter 3]{O'Neill:83}): $\overline{R}(X,Y)Z=\overline{\nabla}_{[X,Y]}Z-[\overline\nabla_X,\overline\nabla_Y]Z$.} of $\mathbb{E}^3(\kappa,\tau)$ is given by
\begin{align}\label{eq:5.2}
	\overline{R}_{R}(X,Y)Z=&(\kappa-3\tau^{2})(\langle X,Z\rangle_{R}Y-\langle Y,Z\rangle_{R}X)\\
	&+\,(\kappa-4\tau^{2})\langle Z,\xi\rangle_{R}(\langle Y,\xi\rangle_{R}X-\langle X,\xi\rangle_{R}Y)\\
	&+\,(\kappa-4\tau^{2})(\langle Y,Z\rangle_{R}\langle X,\xi\rangle_{R}-\langle X,Z\rangle_{R}\langle Y,\xi\rangle_{R})\xi,
\end{align}
for any $X,Y,Z\in\mathfrak{X}(\mathcal{D})$. Analogously, it is easy to check that the curvature tensor of $\mathbb{L}^{3}(\kappa,\tau)$ is given by
\begin{align}\label{eq:5.2.1}
	\overline{R}_{L}(X,Y)Z=&(\kappa+3\tau^{2})(\langle X,Z\rangle_{L}Y-\langle Y,Z\rangle_{L}X)\\
	&-\,(\kappa+4\tau^{2})\langle Z,\xi\rangle_{L}(\langle Y,\xi\rangle_{L}X-\langle X,\xi\rangle_{L}Y)\\
	&-\,(\kappa+4\tau^{2})(\langle Y,Z\rangle_{L}\langle X,\xi\rangle_{L}-\langle X,Z\rangle_{L}\langle Y,\xi\rangle_{L})\xi,
\end{align}
where $X,Y,Z\in\mathfrak{X}(\mathcal{D})$.

Consider again $\Sigma^2$ a non-degenerate surface in $\mathbb{L}^3(\kappa,\tau)$, and let us denote by $\overline{K}_{R}$ and $\overline{K}_{L}$ the sectional curvatures in $\mathbb{E}^{3}(\kappa,\tau)$ and $\mathbb{L}^3(\kappa,\tau)$, respectively, of the non-degenerate tangent plane to $\Sigma^{2}$. Then, given $\widehat{\Sigma}^2_0$ a connected component of $\widehat{\Sigma}^2$ and $\{u_1,u_2\}$ a local orthonormal frame on $(\widehat{\Sigma}^{2}_0,\langle,\rangle_R)$, from~\eqref{eq:2.26} and~\eqref{eq:5.2} it holds
\begin{equation}\label{eq:5.4}
	\overline{K}_{R}=\langle\overline{R}_{R}(u_{1},u_{2})u_{1},u_{2}\rangle_{R}=\tau^{2}+(\kappa-4\tau^{2})\langle N_{R},\xi\rangle_{R}^{2}
\end{equation}
along $(\widehat{\Sigma}^2_0,\langle,\rangle_R)$, which as in Section~\ref{Sec:2} is oriented by the unique unit normal vector field $N_R$ such that $\langle N_R,\xi\rangle_R\geq 0$.
Analogously, if $\{v_1,v_2\}$ is a local orthonormal frame on $(\widehat{\Sigma}^{2}_0,\langle,\rangle_L)$ such that $\langle v_{1},v_{1}\rangle_{L}=1$ and $\langle v_{2},v_{2}\rangle_{L}=-\varepsilon$, it also holds along $(\widehat{\Sigma}^2_0,\langle,\rangle_L)$ that
\begin{equation}\label{eq:5.3}
	\overline{K}_{L}=-\varepsilon\langle\overline{R}_{L}(v_{1},v_{2})v_{1},v_{2}\rangle_{L}=-\tau^{2}-\varepsilon(\kappa+4\tau^{2})\langle N_{L},\xi\rangle_{L}^{2},
\end{equation}
$N_L$ being the unit normal vector field on $(\widehat{\Sigma}^2_0,\langle,\rangle_L)$ such that $\langle N_L,\xi\rangle_L\leq 0$.

Let us assume now that $\kappa\neq -4\tau^2$. Then, from~\eqref{eq:2.24},~\eqref{eq:5.4} and~\eqref{eq:5.3}, we obtain the following relation between both sectional curvatures,
\begin{equation}\label{eq:5.5}
	\overline{K}_{R}=\dfrac{1}{\omega_{L}^{2}}\left(\tau^{2}\left(\omega_{L}^{2}-\varepsilon A\right)-\varepsilon A\overline{K}_{L}\right),
\end{equation}
where
\begin{equation}\label{eq:5.5.1}
	A=\dfrac{\kappa-4\tau^{2}}{\kappa+4\tau^{2}}.
\end{equation}
Let us remark that although equations~\eqref{eq:5.4} and~\eqref{eq:5.3} are obtained over any connected component of $\widehat{\Sigma}^2$,~\eqref{eq:5.5} holds in $\Sigma^2$ by a continuity argument.

From now on, we will denote by $K_e^R:=\textrm{det}(A_R)$ and $K_e^L:=\textrm{det}(A_L)$ the extrinsic curvatures of $\psi:\Sigma^2\rightarrow\mathbb{E}^3(\kappa,\tau)$ and $\psi:\Sigma^2\rightarrow\mathbb{L}^3(\kappa,\tau)$, respectively. The next result establishes a relationship between them.

\begin{proposition}\label{prop:5.1}
	Let $\Sigma^{2}$ be a non-degenerate surface in $\mathbb{L}^{3}(\kappa,\tau)$, and let $\widehat{\Sigma}^2_0$ be a connected component of $\widehat{\Sigma}^2$. Then, the extrinsic curvatures $K_e^R$ and $K_e^L$ in $\widehat{\Sigma}^2_0$ are related by
	\begin{equation}\label{eq:5.6}
		K_{e}^{L}=-\dfrac{\varepsilon}{\omega_{R}^{4}}K_{e}^{R}+\dfrac{4\tau\varepsilon}{\omega_{R}^{4}}\left(\langle A_R(T_R),J_R(T_R)\rangle_R+\tau|T_{R}|_{R}^{2}\right).
	\end{equation}
\end{proposition}

\begin{proof}
	Given $p\in\widehat{\Sigma}^2_0$, let $\{e_1,e_2\}$ be an orthonormal frame on a neighbourhood $\mathcal{U}$ of $p$ with respect to the metric $\langle,\rangle_R$ diagonalizing $A_R$, i.e. such that $A_R(e_i)=\lambda_i^R e_i$, where $\lambda_i^R$ is a smooth function on $\mathcal{U}$ for $i=1,2$. From Proposition~\ref{prop:3.1} we obtain
	\begin{equation}\label{eq:ALei}
		A_L(e_i)=-\frac{\lambda_i^R}{\omega_R}e_i+\frac{2\varepsilon\lambda_i^R}{\omega^3_R}\langle T_R,e_i\rangle_R T_R+\frac{2\varepsilon\tau}{\omega^3_R}\langle J_R(T_R),e_i\rangle_RT_R-\frac{2\tau}{\omega_R}\langle T_R,e_i\rangle_RJ_R(T_R),
	\end{equation}
	for $i=1,2$. Then, writing $A_L(e_i)=\sum_{j=1}^2 a_{ij}^L e_j$, the coefficients are expressed as
	\begin{equation}\label{eq:5.7}
		a_{ii}^{L}=-\dfrac{\lambda^{R}_{i}}{\omega_{R}}+\dfrac{2\varepsilon\lambda^{R}_{i}}{\omega_{R}^{3}}\langle T_{R},e_{i}\rangle_{R}^{2}+\dfrac{2\tau(\varepsilon-\omega_{R}^{2})}{\omega_{R}^{3}}\langle T_{R},e_{i}\rangle_{R}\langle J_{R}(T_{R}),e_{i}\rangle_{R},
	\end{equation}
	and
	\begin{align}\label{eq:5.8}
		a_{ij}^{L}=&\dfrac{2\varepsilon\lambda^{R}_{i}}{\omega_{R}^{3}}\langle T_{R},e_{i}\rangle_{R}\langle T_{R},e_{j}\rangle_{R}+\dfrac{2\varepsilon\tau}{\omega_{R}^{3}}\langle J_{R}(T_{R}),e_{i}\rangle_{R}\langle T_{R},e_{j}\rangle_{R}\\&-\dfrac{2\tau}{\omega_{R}}\langle T_{R},e_{i}\rangle_{R}\langle J_{R}(T_{R}),e_{j}\rangle_{R},
	\end{align}
	for $i,j=1,2$, $i\neq j$. Therefore, after a straightforward computation we get
	\begin{align}\label{eq:5.9}
		{\rm det}(A_{L})=&\dfrac{\lambda^{R}_{1}\lambda^{R}_{2}}{\omega_{R}^{2}}-\dfrac{2\varepsilon\lambda_{1}^{R}\lambda_{2}^{R}}{\omega_{R}^{4}}|T_{R}|_{R}^{2}-\dfrac{2\tau(\varepsilon-\omega_{R}^{2})}{\omega_{R}^{4}}(\lambda_{1}^{R}-\lambda_{2}^{R})\langle T_{R},e_{1}\rangle_{R}\langle T_{R},e_{2}\rangle_{R}\\
		&-\dfrac{4\varepsilon\tau}{\omega_{R}^{4}}(\lambda_{1}^{R}-\lambda_{2}^{R})\langle T_{R},e_{1}\rangle_{R}\langle T_{R},e_{2}\rangle_{R}|T_{R}|^{2}_{R}
		+\dfrac{4\varepsilon\tau^2}{\omega_{R}^{4}}|T_{R}|_{R}^{4}.
	\end{align}
	
	On the one hand, since $T_{R}=\langle T_{R},e_{1}\rangle_{R}e_{1}+\langle T_{R},e_{2}\rangle_{R}e_{2}$, we have
	\begin{equation}\label{eq:5.10}
		\langle A_{R}(T_{R}),J_{R}(T_{R})\rangle_{R}=-(\lambda_{1}^{R}-\lambda_{2}^{R})\langle T_{R},e_{1}\rangle_{R}\langle T_{R},e_{2}\rangle_{R}.
	\end{equation}
	Thus,
	\begin{align}\label{eq:5.11}
		{\rm det}(A_{L})&=\dfrac{\varepsilon}{\omega_{R}^{4}}\left(\varepsilon\omega_R^{2}-2|T_{R}|_{R}^{2}\right){\rm det}(A_{R})+\dfrac{4\tau^{2}\varepsilon}{\omega_{R}^{4}}|T_{R}|_{R}^{4}\\
		&+\dfrac{2\tau}{\omega_{R}^{4}}\langle A_{R}(T_{R}),J_{R}(T_{R})\rangle_{R}\left(\varepsilon-\omega_{R}^{2}+2\varepsilon|T_{R}|^{2}_{R}\right).
	\end{align}
	On the other hand, equations~\eqref{omega_R} and~\eqref{eq:2.26} imply
	\begin{equation}\label{eq:5.12}
		\varepsilon\omega_{R}^{2}-2|T_{R}|_{R}^{2}=-1 \quad\mbox{and}\quad\varepsilon-\omega_{R}^{2}+2\varepsilon|T_{R}|_{R}^{2}=2\varepsilon,
	\end{equation}
	so that
	\begin{equation}\label{eq:5.13}
		{\rm det}(A_{L})=-\dfrac{\varepsilon}{\omega_{R}^{4}}{\rm det}(A_{R})+\dfrac{4\tau\varepsilon}{\omega_{R}^{4}}\langle A_{R}(T_{R}),J_{R}(T_{R})\rangle_{R}+\dfrac{4\tau^{2}\varepsilon}{\omega_{R}^{4}}|T_{R}|_{R}^{4}.
	\end{equation}
	Hence, we get the desired result.
\end{proof}

Finally, we can also consider the Gaussian curvatures $K_R$ and $K_L$ of $(\Sigma^2,\langle,\rangle_R)$ and $(\Sigma^2,\langle,\rangle_L)$, respectively. Let us recall that the well-known Gauss equations of $(\Sigma^{2},\langle\,,\rangle_R)$ and $(\Sigma^2,\langle\,,\rangle_L)$ are given, respectively, by
\begin{equation}\label{eq:5.14}
	K_{R}=\overline{K}_{R}+K_{e}^{R}\quad\mbox{and}\quad K_{L}=\overline{K}_{L}+\varepsilon K_{e}^{L}.
\end{equation}
Thus, on any connected component $\widehat{\Sigma}^2_0$ of $\widehat{\Sigma}^2$ we obtain from~\eqref{eq:5.4} and~\eqref{eq:5.3} that
\begin{equation}\label{eq:5.18.1}
	K_{R}=\tau^{2}+(\kappa-4\tau^{2})\langle N_{R},\xi\rangle_{R}^{2}+K_{e}^{R},
\end{equation}
and
\begin{equation}\label{eq:5.18.2}
	K_{L}=-\tau^{2}-\varepsilon(\kappa+4\tau^{2})\langle N_{L},\xi\rangle_{L}^{2}+ \varepsilon K_{e}^{L}.
\end{equation}
Furthermore, assuming again $\kappa\neq -4\tau^2$, a straightforward computation from~\eqref{eq:5.5} and~\eqref{eq:5.14} yields
\begin{align}\label{eq:5.15}
	K_{R}&=\dfrac{1}{\omega_{L}^{2}}\left(\tau^{2}(\omega_{L}^{2}-\varepsilon A)-\varepsilon A\overline{K}_{L}\right)+K_{e}^{R}\\
	&=-\dfrac{\varepsilon A}{\omega_{L}^{2}}\overline{K}_{L}+\dfrac{(\omega_{L}^{2}-\varepsilon A)\tau^{2}}{\omega_{L}^{2}}+K_{e}^{R}\\
	&=-\dfrac{\varepsilon A}{\omega_{L}^{2}}({K}_{L}-\varepsilon K_e^L)+\dfrac{(\omega_{L}^{2}-\varepsilon A)\tau^{2}}{\omega_{L}^{2}}+K_{e}^{R}\\
	&=-\dfrac{\varepsilon A}{\omega_{L}^{2}}K_{L}+\dfrac{(\omega_{L}^{2}-\varepsilon A)\tau^{2}}{\omega_{L}^{2}}+\dfrac{1}{\omega_{L}^{2}}\left( A K_{e}^{L}+\omega_{L}^{2}K_{e}^{R}\right).
\end{align}
Equivalently,
\begin{equation}\label{eq:5.16}
	\omega_{L}^{2}K_{R}+\varepsilon A K_{L}=(\omega_{L}^{2}-\varepsilon A)\tau^{2}+\omega_{L}^{2}K_{e}^{R}+ A K_{e}^{L}.
\end{equation}

\section{On the geometry of  non-degenerate anisocurved surfaces}\label{sec:5}

We define a non-degenerate \emph{anisocurved} surface in the homogenous space $\mathbb{L}^3(\kappa,\tau)$ as a non-degenerate surface in $\mathbb{L}^3(\kappa,\tau)$ such that it has opposite Gaussian curvature functions $K_R$ and $K_L$ when considered it as an immersion into $\mathbb{E}^3(\kappa,\tau)$ and $\mathbb{L}^3(\kappa,\tau)$, respectively.

As an application of the relations obtained in Section~\ref{sec:curv}, we are presenting some results concerning the geometry of non-degenerate anisocurved surfaces in $\mathbb{L}^3(\kappa,\tau)$, under some extra suitable assumptions on their extrinsic curvatures.

Firstly, let us consider the case where $\Sigma^2$ is a timelike surface in $\mathbb{L}^3(\kappa,\tau)$. Before giving our main results, let us study some particular examples of anisocurved timelike surfaces.

\begin{example}\label{ex:anisocurved}
	Let $\Sigma^2$ be a timelike helix surface immersed into $\mathbb{L}^3(\kappa,\tau)$, and let us compute its Gaussian and extrinsic curvatures in $\mathbb{E}^3(\kappa,\tau)$ and $\mathbb{L}^3(\kappa,\tau)$.
	
	Since $\Sigma^2$ is timelike, $T_R$ is non-zero at any $p\in\Sigma^2$, so we can consider $\{e_1,e_2\}$ a local orthormal frame on $\mathfrak{X}(\Sigma)$ such that $e_1=\frac{T_R}{|T_R|}$ and $e_2=J_R(e_1)$. Furthermore, by Corollary~\ref{cor:constantangle}, $\Sigma^2$ is also an helix surface in $\mathbb{E}^3(\kappa,\tau)$, so $\langle N_R,\xi\rangle_R$ is constant and, consequently,~\eqref{eq:3.11} implies $A_R(T_R)=-\tau J_R(T_R)$. Therefore, $\langle A_R(e_1),e_1\rangle_R=0$ and $\langle A_R(e_1),e_2\rangle_R =-\tau$, so $K_e^R=\textrm{det}(A_R)=-\tau^2$. Analogously, we can easily compute $K_e^L=\textrm{det}(A_L)=\tau^2$. Thus, $K_e^R=-K_e^L=-\tau^2$.  Finally, from~\eqref{eq:5.18.1} and~\eqref{eq:5.18.2} we get $K_R=(\kappa-4\tau^2)\langle N_R,\xi\rangle_R^2$ and $K_L=-(\kappa+4\tau^2)\langle N_L,\xi\rangle_L^2$.
	
	Let us observe that, in particular, Hopf surfaces are anisocurved surfaces satisfying $K_R=-K_L=0$. On the other hand, if $\langle N_R,\xi\rangle_R$ is a non-zero constant, $\Sigma^2$ is an anisocurved timelike helix surface if and only if $(\kappa+4\tau^2)\langle N_L,\xi\rangle_L^2=(\kappa-4\tau^2) \langle N_R,\xi\rangle_R^2$. Or equivalently, in the case $\kappa\neq -4\tau^2$, if and only if $\omega^2_L=A$.
\end{example}

In our main first result we get a nice characterization of Hopf surfaces.

\begin{theorem}\label{teo:5.timelike}
	The only timelike anisocurved surfaces immersed into $\mathbb{L}^{3}(\kappa,\tau)$ with $\kappa+4\tau^{2}>0$ and satisfying $K^{R}_{e}\leq -K^{L}_{e}$ are open pieces of Hopf surfaces.
\end{theorem}

\begin{proof}
	We claim that if $\Sigma^2$ is a timelike surface immersed into $\mathbb{L}^3(\kappa,\tau)$ satisfying the assumptions of the theorem, it verifies $K^R_e=-K^L_e$.
	
	In fact, since $\Sigma^2$ is timelike and anisocurved, $\varepsilon=1$ and $K_R=-K_L$, so we get from~\eqref{eq:5.18.1} and~\eqref{eq:5.18.2} that
	\begin{equation}\label{eq:aux_thm5.1}
		0=(\kappa-4\tau^2)\langle N_R,\xi\rangle^2_R-(\kappa+4\tau^2)\langle N_L,\xi\rangle^2_L+K_e^R+K_e^L,
	\end{equation}
	on any connected component $\widehat{\Sigma}_0^2$ of $\widehat{\Sigma}^2$.
	Observe now that identity~\eqref{eq:2.24} implies that $\langle N_L,\xi\rangle_L^2=\omega_L^2\langle N_R,\xi\rangle_R^2\geq \langle N_R,\xi\rangle_R^2$ and recall that by assumption $\kappa+4\tau^2> 0$. Therefore,~\eqref{eq:aux_thm5.1} derives
	\begin{equation}\label{eq:aux2_thm5.1}
		0\leq -8\tau^2\langle N_R,\xi\rangle^2_R+K_e^R+K_e^L,
	\end{equation}
	so $K_e^R\geq -K_e^L$ in $\widehat{\Sigma}^2$, and by continuity in $\Sigma^2$. Thus, the claim follows from the assumptions of the theorem.
	
	Therefore, denoting $K=K_R=-K_L$ and $K_e=K_e^R=-K_e^L$, equation~\eqref{eq:5.16} becomes
	\begin{equation}\label{eq:5.16_b}
		(\omega_{L}^{2}-A)K=(\omega_{L}^{2}-A)\tau^{2}+(\omega_{L}^{2}-A)K_{e}.
	\end{equation}
	Consequently, either $\omega_{L}^{2}=A$ or $K=\tau^{2}+K_{e}$. In the first case, from~\eqref{eq:2.7} and~\eqref{eq:5.5.1} we derive
	\begin{equation}\label{eq:aux3_thm5.1}
		\langle N_{L},\xi\rangle_{L}^{2}=-\dfrac{4\tau^{2}}{\kappa+4\tau^{2}}\leq 0,
	\end{equation}
	so necessarily $\tau=0$ and $\langle N_L,\xi\rangle_L=0$. Otherwise, $K=\tau^{2}+K_{e}$. Then~\eqref{eq:5.18.2} also yields $\langle N_{L},\xi\rangle_{L}=0$. Consequently, from~\eqref{eq:2.24} it also holds $\langle N_R,\xi\rangle_R=0$, so $\Sigma^{2}$ is locally isometric to a Hopf surface.
\end{proof}

In the case where $\kappa+4\tau^2<0$ and $K^R_e=-K^L_e$, we can also characterize timelike anisocurved helix surfaces.

\begin{proposition}\label{prop:5.timelike}
The only timelike anisocurved surfaces immersed into $\mathbb{L}^{3}(\kappa,\tau)$ with $\kappa+4\tau^{2}<0$ and satisfying $K^{R}_{e}= -K^{L}_{e}$ are open pieces of Hopf surfaces or of timelike helix surfaces such that $\omega^2_L=A$.
\end{proposition}
\begin{proof}
	Proceeding as in Theorem~\ref{teo:5.timelike}, either $\omega^2_L=A$ or $K=\tau^2+K_e$, concluding in both cases that $\langle N_L,\xi\rangle_L$ is constant, so by definition $\Sigma^2$ has constant angle in $\mathbb{L}^3(\kappa,\tau)$. The result follows from Example~\ref{ex:anisocurved}.
\end{proof}

Let us see now that we can obtain nice consequences of the above theorems when considering particular cases of ambient spaces. On the one hand, in the case where $\mathbb{L}^3(\kappa,\tau)$ is a Lorentzian Berger sphere or the Lorentzian Heisenberg space, the condition $\kappa+4\tau^{2}>0$ is trivially satisfied. Therefore,
\begin{corollary}\label{cor:5.1}
	The only timelike anisocurved surfaces immersed into the Lorentzian Berger sphere $\mathbb{S}^{3}_{b,1}(\kappa,\tau)$ or into the Lorentzian Heisenberg space $Nil^3_1(\tau)$ satisfying $K^{R}_{e}\leq -K^{L}_{e}$ are open pieces of Hopf surfaces.
\end{corollary}

On the other hand, observe that the condition $\kappa+4\tau^2<0$ can only be satisfied for the product spaces $M^2(\kappa)\times\mathbb{R}$ with $\kappa<0$ and for the universal cover of the special linear group, $\widetilde{Sl_2}(\mathbb{R})(\kappa,\tau)$ when $\kappa<-4\tau^2$. Thus,
\begin{corollary}\label{cor:5.3}
	The only timelike anisocurved surfaces immersed into the universal cover of the Lorentzian special linear group $\widetilde{Sl_{2,1}}(\mathbb{R})(\kappa,\tau)$, with $\kappa<-4\tau^2$, and satisfying $K_e^R=-K_e^L$ are open pieces of a Hopf surface or of a timelike helix surface such that $\omega_L^2=A$.
\end{corollary}

As a final particular ambient space, let us consider the case $\tau=0$, i.e. let us assume that the ambient is a product space $M^2(\kappa)\times\mathbb{R}_1$. Let us observe that Hopf surfaces in $M^2(\kappa)\times\mathbb{R}_1$ are just cylinders over a regular curve in $M^2(\kappa)$, and from Example~\ref{ex:anisocurved} they satisfy $K_R=K_L=K_e^R=K_e^L=0$. It is worth pointing out that Barbosa and do Carmo provided in~\cite{Barbosa:20} a really nice characterization of complete cylinders in the Riemannian product $\mathbb{H}^2(\kappa)\times\mathbb{R}$ as the only complete and connected surfaces such that $K=K_R=0$.  Let us see first that it is possible to characterize cylinders in a similar way as the results above.

\begin{corollary}
	The only timelike anisocurved surfaces immersed into the Lorentzian product space $M^{2}(\kappa)\times\mathbb{R}_{1}$, $\kappa>0$ ($\kappa<0$, respectively), such that $K^{R}_{e}\leq -K^{L}_{e}$ ($K^{R}_{e}\geq -K^{L}_{e}$, respectively) are open pieces of cylinders over a regular curve of $M^{2}(\kappa)$.
\end{corollary}

\begin{proof}
	The case $\kappa>0$ follows immediately from Theorem~\ref{teo:5.timelike}. In the case $\kappa<0$, we can proceed in an analogous way as in the proof of Theorem~\ref{teo:5.timelike} to conclude that $K^R_e=-K^L_e$, and then the result follows from Proposition~\ref{prop:5.timelike} since $A=1$ for any $\mathbb{L}^3(\kappa,0)$.
\end{proof}

Moreover, we can also characterize cylinders in $M^2(\kappa)\times\mathbb{R}_1,$ with $\kappa\neq 0$, as the only non-degenerate anisocurved surfaces with constant Gaussian curvature.

\begin{theorem}\label{prop:5.2}
	Let $\Sigma^2$ be a non-degenerate anisocurved surface immersed into the Lorentzian product space $M^{2}(\kappa)\times\mathbb{R}_1$, with $\kappa\neq 0$. Then,
	\begin{equation}\label{eq:K_prod}
		K=\frac{\kappa\,(\omega^2_L-\varepsilon)}{2\,(\omega^2_L+\varepsilon)},
	\end{equation}	
where $K=K_R=-K_L$. Furthermore, $K$ is constant if and only if $\Sigma^2$ is a piece of a cylinder over a regular curve of $M^2(\kappa)$.
\end{theorem}
\begin{proof}
	Since $\tau=0$, $A=1$, and from Proposition~\ref{prop:5.1} we get $K_e^R=-\varepsilon \omega^4_R K_e^L$. Thus, equation~\eqref{eq:5.16} reads
	\begin{equation}\label{eq:nuevo_prod}
		\frac{1}{\omega^2_L}\left(\omega^2_L-\varepsilon\right)\left(\omega^2_LK-K_e^L\right)=0.
	\end{equation}
Consequently, either $\omega^2_L=\varepsilon$ or $K_e^L=\omega^2_L K$. However, the first situation can only hold when $\varepsilon=1$ and $\omega^2_L=1$, or equivalently $\langle N_L,\xi\rangle_L=0$. Consequently, $\Sigma^2$ is a piece of a cylinder over a regular curve of $M^2(\kappa)$. Furthermore, since in this case $K=K_e^R=0$, it is also satisfied that  $K_e^L=\omega^2_L K$. Thus, the second identity necessarily holds, so~\eqref{eq:5.18.2} becomes
\begin{equation}\label{eq:nuevo_prod_2}
	K=\varepsilon\kappa\langle N_L,\xi\rangle_L^2-\varepsilon \omega^2_L K,
\end{equation}	
and from~\eqref{eq:2.7} it immediately yields that
\begin{equation}\label{eq:nuevo_prod_3}
	2K(\omega^2_L+\varepsilon)=\kappa(\omega^2_L-\varepsilon).
\end{equation}
Finally,~\eqref{eq:K_prod} follows by observing that, under the assumptions of the theorem, $\omega_L^2+\varepsilon$ cannot vanish. In fact, it could only vanish if $\Sigma^2$ were spacelike and $\omega_L^2=1$, but in this case $\Sigma^2$ would be a slice from~\eqref{eq:2.7}. However, slices satisfy $K_R=K_L=\kappa$, so they are not anisocurved surfaces except when $\kappa=0$.

In order to prove the last assertion of the theorem, we just have to observe that if $K$ is constant, $\omega_L$ is also constant, so $\Sigma^2$ is a non-degenerate anisocurved helix surface in $M^2(\kappa)\times\mathbb{R}_1$. Then, following a similar reasoning as in Example~\ref{ex:anisocurved}, $K_e^R=K_e^L=0$ and $\langle N_L,\xi\rangle_L^2=\langle N_R,\xi\rangle_R^2$, so from~\eqref{eq:2.24} $\omega_L^2=1$ and $\Sigma^2$ is necessarily a piece of a cylinder.
\end{proof}	

As it has been remarked in the proof of Theorem~\ref{prop:5.2}, slices are anisocurved surfaces if and only if $M^2(\kappa)\times\mathbb{R}_1=\mathbb{L}^3$. This fact motivates the following non-existence result for spacelike surfaces.

\begin{theorem}\label{teo:5.spacelike}
	There do not exist any anisocurved spacelike surface in $\mathbb{L}^3(\kappa,\tau)$, $\kappa>0$, such that $K_e^R\geq K_e^L$.
\end{theorem}
\begin{proof}
	Let us assume that there exists a spacelike surface $\Sigma^2$ in $\mathbb{L}^3(\kappa,\tau)$ under the assumptions of the theorem. Then, following an analogous argument as in the proof of Theorem~\ref{teo:5.timelike} we can conclude that $K_e^R=K_e^L$. In fact, since $\Sigma^2$ is spacelike we have $\varepsilon=-1$, and we can easily derive from the assumptions of the theorem,~\eqref{eq:5.18.1} and~\eqref{eq:5.18.2} that
	\begin{equation}\label{eq:aux1_thm5.2}
		2\kappa \langle N_R,\xi\rangle_R^2+K_e^R-K_e^L \leq 0,
	\end{equation}	
	so $K_e^R\leq K_e^L$, which jointly with the assumption of the theorem implies $K_e^R=K_e^L$.
	
	Taking into account now that $\varepsilon=-1$, $K_R=-K_L=K$ and $K_e^R=K_e^L=K_e$, equation~\eqref{eq:5.16} becomes
	\begin{equation}\label{eq:aux2_thm5.2}
		(\omega_L^2+A)(K-\tau^2-K_e)=0.
	\end{equation}
	On the one hand, condition $K=\tau^2+K_e$ cannot hold since it would imply $\Sigma^2$ being a Hopf surface, and therefore timelike. On the other hand, if $\omega^2_L+A=0$, then
	\begin{equation}
		\omega_L^2=-A=\frac{4\tau^2-\kappa}{4\tau^2+\kappa},
	\end{equation}
	which contradicts the fact that $\omega_L^2\geq 1$. Thus,~\eqref{eq:aux2_thm5.2} leads to a contradiction and the result follows.
\end{proof}

\section*{Acknowledgements}
The first author is partially supported by the Spanish MICINN project PID2021-126217NB-I00. The second author is also partially supported by CNPq, Brazil, under the grant 311124/2021-6 and Propesqi (UFPE).

\bibliographystyle{amsplain}

\bigskip

\address{% First Author
	Departamento de Matem\'aticas \\
	Universidad de C\'ordoba \\
	14071 Campus de Rabanales, C\'ordoba\\
	Spain
}
{alma.albujer@uco.es}

\address{% Second Author
	Departamento de Matem\'atica \\
	Universidade Federal de Pernambuco \\
	50.740-540 Recife, Pernambuco\\
	Brazil
}
{fabio.reis@ufpe.br}
\end{document}